\documentclass[11pt]{amsart}
\usepackage{amssymb,latexsym}
\theoremstyle{definition}
\newtheorem{theorem}{Theorem}
\newtheorem{corollary}[theorem]{Corollary}
\newtheorem{proposition}[theorem]{Proposition}
\newtheorem{lemma}[theorem]{Lemma}
\newtheorem{definition}[theorem]{Definition}
\newtheorem{example}[theorem]{Example}
\newtheorem{notation}[theorem]{Notation}
\newtheorem{remark}[theorem]{Remark}

\date{}
\newcommand{\numberset}{\mathbb}

\newcommand{\N}{\numberset{N}}

\newcommand{\F}{\numberset{F}}

\newcommand{\Pro}{\numberset{P}}

\newcommand{\Ol}{\mathcal{O}}

\newcommand{\mC}{\mathcal{C}}
\newcommand{\mL}{\mathcal{L}}
\newcommand{\mV}{\mathcal{V}}
\newcommand{\mB}{\mathcal{B}}

\usepackage{hyperref}

\usepackage[margin=2.6cm]{geometry}

\usepackage{mathptmx}

\usepackage{amsaddr}

\begin{document}

\title[]{On the duals of geometric Goppa codes from norm-trace curves}

\author{Edoardo Ballico$^1$}
\address{Department of Mathematics, University of Trento\\Via Sommarive 14,
38123 Povo (TN), Italy}
\email{$^1$edoardo.ballico@unitn.it}

\author{Alberto Ravagnani$^2$ $^*$}
\address{Department of Mathematics, University of Neuch\^{a}tel\\Rue
Emile-Argand 11, CH-2000 Neuch\^{a}tel, Switzerland}
\email{$^2$alberto.ravagnani@unine.ch}

\thanks{$^1$Partially supported by MIUR and GNSAGA}
\thanks{$^*$Corresponding author}

\maketitle

\providecommand{\bysame}{\leavevmode\hbox to3em{\hrulefill}\thinspace}

\begin{abstract}
  In this paper we study the dual codes of a wide family of
evaluation codes on
norm-trace curves. We explicitly find out their minimum distance and give
a lower bound for the number of their minimum-weight codewords.
A general geometric approach is performed and applied to study in particular the
dual codes of
 one-point and two-point codes arising from norm-trace curves through Goppa's
construction,
 providing in many cases their minimum
distance and some bounds on the number of their minimum-weight codewords.
The results are obtained by showing that the supports of the minimum-weight
codewords of the studied codes obey some precise geometric laws as
zero-dimensional subschemes of the projective plane. Finally, the dimension of
some classical two-point Goppa codes on norm-trace curves is explicitely
computed.
\end{abstract}

\section{Introduction} \label{intr}
Let $r \ge 2$ be an integer and let $q$ denote a prime power (fixed). Consider
the field
extension $\F_q \subseteq \F_{q^r}$ and denote by $\Pro^2$ the projective plane
defined
over the field $\F_{q^r}$. Write $c:=\frac{q^r-1}{q-1}$ and denote by $Y_r
\subseteq \Pro^2$ the curve having
$$x^c=y^{q^{r-1}}+y^{q^{r-2}}+\cdots +y^q+y$$
as an affine equation. Denote by $\mbox{Tr}_r:\F_{q^r} \to \F_q$ and
$\mbox{N}_r:\F_{q^r} \to \F_q$ the $\F_q$-linear maps (named \textbf{trace} and
\textbf{norm}, respectively) defined by
$$\mbox{Tr}_r (\alpha):=\alpha^{q^{r-1}}+ \alpha^{q^{r-2}}+ \cdots + \alpha, \ \
\ \ \mbox{N}_r(\alpha):=\alpha^c, \ \ \ \ \ \mbox{for any } \alpha \in
\F_{q^r}.$$
The curve $Y_r$ is in fact defined by the equation
$\mbox{N}_r(x)=\mbox{Tr}_r(y)$ and so it is called the \textbf{norm-trace} curve
associated to the integer $r$.
If $r=2$ then $Y_2$ is the well-known Hermitian curve. We studied the geometric
properties of the dual codes of Goppa codes on $Y_2$ in \cite{br1}, \cite{br2}
and \cite{br3}.
 Here we focus on the more complicated case $r \ge 3$.
In this situation the curve $Y_r$ is singular. The only point at infinity, of
projective coordinates $P_\infty:=(0:1:0)$, is also the only singular point of
the curve (straightforward computation). Denote by $\pi:C_r \to Y_r$ the
normalization, which is known to be a bijection. The genus of $Y_r$ (which is by
definition the genus of $C_r$) is $g=(q^{r-1}-1)(c-1)/2$ and the Weierstrass
semigroup
associated to $P_\infty$ is well studied in \cite{geo} and known to be
$$H(P_\infty)= \langle q^{r-1},c \rangle.$$
The curve $Y_r$ carries $|Y_r(\F_{q^r})|=q^{2r-1}+1$ rational points and we have
already
stated that $q^{2r-1}$ of them lie in the affine chart $\{ z \neq 0 \}$. Let
$Q_\infty:=\pi^{-1}(P_\infty)$. For any $0 \le s \le cq^r$ a basis of the
Riemann-Roch
space $L(sQ_\infty)$ is formed by the (pull-backs of the) monomials
$$\{ x^iy^j : \  \ i<q^r, \ \ j<q^{r-1}, \ \ iq^{r-1}+jc \le s \}$$
(see \cite{geil}). Since for any prime power $q$ and for any $r \ge 2$ we get
$(q^r-1)/(q-1)
>q^{r-1}$, the degree of $Y_r$ is exactly $c=(q^r-1)/(q-1)$. 
 The pull-backs of the monomials $\{ 1,x,y\}$ span the vector space
$H^0(C_r,\pi^*(\Ol_{Y_r}(1)))$, which is contained into $L(cQ_\infty)$.
Since we know $\dim_{\F_{q^r}} L(cQ_\infty)=3$, we get exactly
$L(cQ_\infty)=H^0(C_r,\pi^*(\Ol_{Y_r}(1)))$, the vector space of the homogeneous
degree $1$ forms on the curve $Y_r$ (we pull-back forms through $\pi$ in order
to work on a smooth curve).
More generally, if $0<d<q$ then the vector space of the degree $d$ homogeneous
forms on the curve $Y_r$, $H^0(C_r,\pi^*(\Ol_{Y_r}(d)))$, is exactly
$L(dcQ_\infty)$ and we will widely use this geometric fact in the paper to get a
bond between classical Goppa codes and a new class of evaluation codes. For any
$0<d<q^{r-1}$ a natural basis for the vector space
$H^0(C_r,\pi^*(\Ol_{Y_r}(d)))$ of the degree $d$ homogeneous forms on the curve
$Y_r$ is made of the monomials $x^iy^j$ such that $i,j \ge 0$ and $i+j \le d$
(up to a homogeneization). Indeed, these monomials are linearly independent
because they appear also in the cited basis of $L(dcQ_\infty)$. In general we
have an inclusion of vector spaces
$$H^0(C_r,\pi^*(\Ol_{Y_r}(d))) \subseteq L(dcQ_\infty).$$

\section{One-point codes: a first analysis} \label{sec1}
In this section we study a simple family of evaluation codes on $Y_r$ curves.
The method will be improved at a second time.
First of all, we state a technical result.

\begin{lemma}\label{u00.01}
Fix integers $d>0$, $z\ge 2$ and a zero-dimensional scheme $Z\subseteq \mathbb
{P}^2$ such that
$\deg (Z) =z$.
\begin{enumerate}
\item[(a)] If $z\le d+1$, then $h^1(\mathbb {P}^2,\mathcal {I}_Z(d))=0$.

\item[(b)]\label{strano1} If $d+2 \le z\le 2d+1$, then $h^1(\mathbb
{P}^2,\mathcal {I}_Z(d))>0$ if and only if there
is a line $L$ such that $\deg (L\cap Z)\ge d+2$.
\end{enumerate}
\end{lemma}
\begin{proof}
 See \cite{br1}, Lemma 2.
\end{proof}

\begin{definition} \label{def1}
 Let $0<d<q^{r-1}-1$ be an integer. Set $B:= Y_r \setminus \{ P_\infty \}$. Then
$\mathcal{C}(d)$ will denote the linear code obtained evaluating the vector
space
$H^0(C_r,\pi^*(\Ol_{Y_r}(d)))$ on $\pi^{-1}(B)$.
\end{definition}

\begin{notation}
 By the injectivity of $\pi$, from now to the end of the paper we will write $S$
instead of $\pi^{-1}(S)$, for \textit{any} $S \subseteq Y_r(\F_{q^r})$.
\end{notation}

\begin{remark}
If $0<d<q$ then the code $\mathcal{C}(d)$ is the so-called one-point code
$\mC_s$ ($s:=dc$) on $Y_r$ obtained by
evaluating $L(sP_\infty)$ at the rational points of the curve different from
$P_\infty$ (see Section \ref{intr}). For any $0<d<q^{r-1}-1$ we have an
inclusion of codes $\mathcal{C}(d) \subseteq \mC_s$ (the curve $Y_r$ is not in
general projectively normal) which gives
$\mathcal{C}(d)^\perp \supseteq \mC_s^\perp$. Hence the minimum distance of
$\mC_s^\perp$ is at least the minimum distance of $\mathcal{C}(d)^\perp$
(studied below).
\end{remark}

\begin{theorem}\label{te1}
 The minimum distance of a $\mathcal{C}(d)^\perp$ code is
$d+2$. Moreover, the points in the support of a minimum-weight codewords are
collinear. If $q \le d<q^{r-1}-1$ then the support of a minimum-weight codeword
of $\mathcal{C}(d)^\perp$ is
contained into a line which cannot be horizontal.
\end{theorem}

\begin{proof}
  Consider the line $L$ of equation $x=0$. By the properties of the trace map
the equation $\mbox{Tr}_r(y)=0$ has exactly $q^{r-1}$ distinct solutions,
i.e. $|Y_r(\F_{q^r}) \cap L|=q^{r-1}$. Since $d \le q^{r-1}-2$ we can pick out
$d+2$ distinct affine points $$P_1=(0,y_1), ...,P_{d+2}=(0,y_{d+2})$$ from this
intersection. They are obviously different from $P_\infty$. The natural
parity-check
matrix of $\mathcal{C}(d)^\perp$ has at most $d+1$ non-zero rows (those
associated to the monomials $1,y,...,y^d$). Hence the columns associated to the
points $P_1,...,P_{d+2}$ are linearly dependent, i.e. $\{ P_1,...,P_{d+2}\}$
contains the support of a codeword of $\mathcal{C}(d)^\perp$ of weight $w \le
d+2$. It follows that the minimum distance of $\mathcal{C}(d)^\perp$ is smaller
or equal than $d+2$.
Since $0<d<q^{r-1}-1$ we have in particular $d<c=\deg(Y_f)$. Hence the
restriction (and pull-back) map
$$\rho_d:H^0(\Pro^2,\mathcal{O}_{\Pro^2}(d)) \to H^0(C_r,\pi^*(\Ol_{Y_r}(d)))$$
is injective. Let $S$ be the support of a minimum-weight codeword. The set $S$
imposes dependent conditions to $H^0(C_r,\pi^*(\Ol_{Y_r}(d)))$; moreover, no
proper subset $S' \varsubsetneq S$ imposes dependent conditions to that space.
Hence the minimum distance of $\mathcal{C}(d)^\perp$ is exactly $\sharp(S)$. We
already know that $\sharp(S) \le d+2$. The set $S$ imposes of course dependent
conditions also to the image of $\rho_d$. Since this linear map is injective, we
get that  $S$ imposes dependent conditions also to
$H^0(\Pro^2,\mathcal{O}_{\Pro^2}(d))$, i.e. $h^1(\Pro^2,\mathcal{I}_S(d))>0$. By
Lemma \ref{u00.01} we must have that $\sharp(S) \ge d+2$. Hence $\sharp(S)=d+2$
is the minimum distance of $\mathcal{C}(d)^\perp$. Lemma \ref{u00.01} implies
also that $d+2$ points in the support of a minimum-weight codewords have to be
collinear.

Let us prove the second part of the statement. If $d \ge q$ then
$x^i \in L(dcP_\infty)$ for any $i=0,1,...,d+1$ (while if $d<q$ we do not have
$x^{d+1}$ in $L(dcP_\infty)$ as a monomial). If $q \le d <q^{r-1}-1$ then the
minimum distance of $\mC(d)^\perp$ is again $d+2$
(reached in any case on vertical lines) but $d+2$ columns associated to $d+2$
points lying on a horizontal line are in fact always linearly
\textit{independent} (one can immediately find a Vandermonde submatrix of rank
$d+2$).
\end{proof}

\begin{theorem}\label{te2}
 The number of the minimum-weight codewords of a
$\mathcal{C}(d)^\perp$ code is at least $$(q^r-1)\left[q^r \binom{q^{r-1}}{d+2}+
(q^r-1)\binom{\frac{q^r-1}{q-1}}{d+2} \right].$$
\end{theorem}

\begin{proof}
 By Theorem \ref{te1} we know that the minimum distance of
$\mathcal{C}(d)^\perp$ is $d+2$ and that the points of the support of a
minimum-weight codeword are collinear. Pick out any $\alpha \in \F_{q^2}$ and
consider the line $L_\alpha$ of equation $x=\alpha$. The equation
$\mbox{Tr}_r(y)=\alpha$ has $q^{r-1}$ distinct solutions. Choose any distinct
affine
$d+2$ points $P_1,...,P_{d+2}$ in the intersection $Y_r(\F_{q^r}) \cap
L_\alpha$. The parity-check matrix of the code $\mathcal{C}(d)^\perp$ has at
most $d+1$ linearly independent rows (those associated to the monomials
$1,y,...,y^d$) and so there exist a dependent relation among the columns
associated to the points $P_1,...,P_{d+2}$, i.e. $\{P_1,...,P_{d+2}\}$ is the
support of a minimum-weight codewords of $\mathcal{C}(d)^\perp$ ($d+2$ is known
to be the minimum distance). In $H^0(C_r,\pi^*(\Ol_{Y_r}(d)))$ we have only
monomials $x^iy^j$ with the property $i \le d$. Hence we can repeat the proof
with horizontal lines and the norm map. In this case we can choose any line of
the form $y=\alpha$, provided that $\alpha \neq 0$.  The lower bounds in the
statement follow.
\end{proof}

\begin{remark}
If $d<q$ then Theorem \ref{te2} describes in fact one-point codes on norm-trace
curves.
 Indeed, by setting $s:=dc$ we get an identity of vector spaces
$L(sP_\infty)=H^0(C_r,\pi^*(\Ol_{Y_r}(d)))$ and so the one-point code $\mC_s$
obtained evaluating $L(sP_\infty)$ on $Y(\F_{q^r})\setminus \{ P_\infty\}$ is in
fact $\mC(d)$. This proves the following result.
\end{remark}

\begin{corollary}
 Let $s \ge 0$ be an integer. Write $s=dc-a$ with $0 \le a \le c-1$. Assume
$0<d<q^{r-1}-1$. The dual minimum distance of the one-point code $\mC_s$
obtained evaluating the vector space $L(sP_\infty)$ on $Y_r(\F_{q^r})\setminus
\{ P_\infty\}$ is $d+2$. If $d<q$ then the number of the minimum-weight
codewords of $\mathcal{C}_s^\perp$ code is at least $(q^r-1)\left[q^r
\binom{q^{r-1}}{d+2}+
(q^r-1)\binom{\frac{q^r-1}{q-1}}{d+2} \right]$.
\end{corollary}
\begin{proof}
 The minimum distance of $\mC_s^\perp$ is at least the minimum distance of
$\mC(d)^\perp$, which is $d+2$. Since in $L(sP_\infty)$ we have only the
monomials $y^i$ with $i \le d$ this weight is reached on vertical lines as in
the proof of Theorem \ref{te1}. If $d<q$ then apply Theorem \ref{te2}.
\end{proof}

\begin{example}\label{ex1}
 Set $q:=2, r:=3$ and $d:=2$. The code $\mC(d)^\perp$ can be studied by writing
a simple \texttt{Magma} program. The minimum distance is $4$. If $d:=1$ then
$\mC(d)$ has dual
minimum distance $3$ and the number of the minimum-weight codewords of
$\mathcal{C}(d)^\perp$ is $3360$.
\end{example}

\section{A few remarks on Goppa codes} \label{si}
Let $q$ be a prime power and let $\Pro^k$ be the projective space of dimension
$k$ over the field
 $\F_q$. Consider a smooth curve $X\subseteq \Pro^k$  and a divisor $D$ on it. 
Take points $P_1,...,P_n \in X(\F_q)$ avoiding the support of $D$ and
set
 $\overline{D}:=\sum_{i=1}^n P_i$.   
The code $\mathcal{C}(\overline{D},D)$ is defined to be the code obtained
evaluating the vector 
space $L(D)$ at the points $P_1,...,P_n$ (see \cite{Ste}). 
These codes were introduced in 1981 by Goppa, who was interested in studying
their dual codes. 
Since a norm-trace curve $Y_r$ is not a smooth curve, 
when writing ``~Goppa code on $Y_r$~''  we mean ``~Goppa code on $C_r$~'' 
(the normalization of $Y_r$). The points of $Y_r$ will be identified with those
of $C_r$ through 
the injectivity of the normalization $\pi:C_r \to Y_r$.

\begin{definition}\label{si}
 Let $q$ be a prime power. We say that codes $\mathcal{C}, \mathcal{D}$ on the
same field $\F_q$ and of the same lenght are \textbf{strongly isometric}
if there exists a vector ${\bf{x}}=(x_1,...,x_n)\in \F_q^n$ of non-zero
components 
such that $$\mathcal{C}={\bf{x}}\mathcal{D}:=
 \{(x_1v_1,...,x_nv_n)\in \F_q^n \ \mbox{ s.t. } \ (v_1,...,v_n) \in \mathcal{D}
\}.$$
The notation will be $\mathcal{C} \sim \mathcal{D}$ and this defines of course
an equivalence relation.
\end{definition}
\begin{remark}\label{duuu}
 Take the setup of Definition \ref{si}. Then $\mathcal{C} \sim \mathcal{D}$
 if and only if $\mathcal{C}^\perp \sim \mathcal{D}^\perp$. Indeed, if
$\mathcal{C}={\bf{x}}\mathcal{D}$
then $\mathcal{C}^\perp={\bf{x}}^{-1}\mathcal{D}^\perp$, where
${\bf{x}}^{-1}:=(x_1^{-1},...,x_n^{-1})$.
A strongly isometry of codes preserves in fact the minimum distance of a code,
its weight distribution
and the supports of its codewords. 
\end{remark}

\begin{remark}\label{ct}
 Take the setup of the beginning of the section. Let $D$ and $D'$ be divisors on
$X$ and take points 
$P_1,...,P_n \in X(\F_q)$ 
avoiding both the supports of $D$ and $D'$.
 Set $\overline{D}:=\sum_{i=0}^n P_i$. It is known (see \cite{mp}, Remark 2.16)
that if 
$D\sim D'$ (as divisors)
 then  $\mathcal{C}(\overline{D},D) \sim \mathcal{C}(\overline{D},D')$.
By Remark \ref{duuu} we have also $\mathcal{C}(\overline{D},D)^\bot \sim
\mathcal{C}(\overline{D},D')^\bot$.
\end{remark}

\section{One-point codes}

\begin{definition}\label{def_a}
 Let $0<d<q^{r-1}-1$ and $a \ge 0$ be integers. We denote by $\mC(d,a)$ the code
obtained evaluating $H^0(C_r,\pi^{*}(\Ol_{Y_r}(d)(-aP_\infty)))$ on the set
$B:=Y_r(\F_{q^r})\setminus \{ P_\infty \}$.
\end{definition}

\begin{theorem} \label{teo_a}
 Let $\mC(d,a)$ be as in Definition \ref{def_a}. Assume $a=1$. Then the minimum
distance
of $\mC(d,a)^\perp$ is $d+1$ and the number of the minimum-weight codewords of
$\mC(d,a)^\perp$ is exactly $(q^r-1)q^r\binom{q^{r-1}}{d+1}$.
\end{theorem}

\begin{proof}
 Since $0<a \le
d$ if a monomial $x^iy^j$ is in the vector space
$H^0(C_r,\pi^*(\Ol_{Y_r}(d)(-aP_\infty)))$ then we must have $j \le d-1$ (we
work up to a homogeneization). On the other hand,
$1,y,...,y^{d-1}$ are in any case in this space. As in the proof of Theorem
\ref{te1}, any $d+1$
affine points in the intersection of $Y_r(\F_{q^r})$ and a vertical line of
equation $x=\alpha$ contain the support of a codeword of $\mC(d,a)^\perp$. Hence
the minimum distance of $\mC(d,a)^\perp$ is at most $d+1$. Let $S\subseteq
Y_r(\F_{q^r})$ be the support of a minimum-weight codeword of $\mC(d,a)^\perp$.
The minimum distance of this code is exactly $\sharp(S)$. Since $d<q^{r-1}-1<c$
the restriction (and pull-back) map
$$\rho_{d,a}:H^0(\Pro^2,\mathcal{O}_{\Pro^2}(d)(-aP_\infty)) \to
H^0(C_r,\pi^*(\Ol_{Y_r}(d)(-P_\infty)))$$
 is injective. Since $S$ imposes dependent conditions to the vector space 
$H^0(C_r,\pi^*(\Ol_{Y_r}(d)(-P_\infty)))$ then it has to impose dependent
conditions also to $H^0(\Pro^2,\mathcal{O}_{\Pro^2}(d)(-P_\infty))$, i.e.
$h^1(\Pro^2,\mathcal{I}_{P_\infty \cup
S}(d))>h^1(\Pro^2,\mathcal{I}_{P_\infty}(d))$. In particular we have
$h^1(\Pro^2,\mathcal{I}_{P_\infty \cup S})>0$. Observe
that $\sharp(S)+a \le d+1+1=d+2$. By Lemma \ref{u00.01} we get the existence of
a line $L \subseteq \Pro^2$ such that $\deg(L \cap (P_\infty \cup S)) \ge d+2$.
Since $\sharp(S) \le d+1$ we deduce $P_\infty \subseteq L$ (as schemes). Hence
$L$ is either the line at infinity, or a vertical line. The line at infinity
does not intersect $Y_r$ at any affine point, so $L$ has to be a vertical line.
It follows 
$$\sharp(S) \ge \deg(L \cap S) \ge d+2-\deg(L \cap P_\infty)=d+2-1=d+1.$$
Since we have shown that $\sharp(S) \le d+1$,  the minimum distance of
$\mC(d,a)^\perp$ is exactly $d+1$ and $S$ consists of $d+1$ points on a vertical
line.
\end{proof}

\begin{corollary}\label{one}
Let $\mathcal{C}_s$ be the one-point code on $Y_r$ obtained evaluating the
vector space $L(sP_\infty)$ on the rational points of $Y_r$ different from
$P_\infty$. Divide $s$ by $c$ with remainder and write $s=dc-a$ with $0 \le a
\le c-1$. Assume $0 < d < q^{r-1}-1$ and $a \le d$.
\begin{enumerate}
 \item If $a=0$ then the minimum distance of $\mC_s^\perp$ is $d+2$.
\item If $a=1$ then the minimum distance of $\mC_s^\perp$ is $d+1$.
\item If $1<a \le d$ then the minimum distance of $\mC_s^\perp$ is at least
$d+2-a$ and at most $d+1$.
\end{enumerate} 
\end{corollary}

\begin{proof}
Since $s=dc-a$ we have a linear equivalence $sP_\infty \sim
dcP_\infty-aP_\infty$. Since $0<d<q^{r-1}-1$ the minimum distance of
$\mC_s^\perp$ is at least the minimum distance of $\mC(d,a)^\perp$, because of
the inclusion
$$H^0(C_r,\pi^*(\Ol_{Y_r}(d))(-aP_\infty))\subseteq L(sP_\infty).$$ If $a \in \{
0,1\}$ then as in the proof of Theorem \ref{te1} and Theorem \ref{teo_a} this
minimum distance is reached on vertical lines (the monomials of the form $y^i$
appearing in $L(sP_\infty)$ and in $H^0(C_r,\pi^*(\Ol_{Y_r}(d))(-aP_\infty))$
are the same). If $1<a \le d$ then we can repeat the proof of Theorem
\ref{teo_a} into a slightly general context.
\end{proof}

\begin{remark}
It could be pointed out that Corollary \ref{one} describes in fact also some
classical Goppa one-point codes arising from norm-trace curves
(and not only the dual codes of such kind of codes). Indeed, norm-trace curves
turn out to be a particular case of Castle curves and so (\cite{cast},
Proposition 5) we get a strong isometry of one-point codes $\mC_s^\perp \sim
\mC_{n+2g-2-s}$, in the sense of Definition \ref{si}, with $n=q^{2r-1}$ and
$2g-2=(q^{r-1}-1)(c-1)-2$. It follows that the metric properties of 
$\mC_{n+2g-2-s}$ are those of $\mC_s^\perp$.
\end{remark}

\section{Two-point codes}\label{two}

Let $P_0$ denote the point of $Y_r$ of projective coordinates $(0:0:1)$. In this
section we study codes obtained by using zero-dimensional plane schemes
supported by
$P_\infty$ and $P_0$. The results can be applied to study several two-point
codes on norm-trace curves (as we will explain in details).

\begin{definition}\label{def_ab}
 Let $0<d<q^{r-1}-1$ be an integer. Choose integers $a,b \ge 0$. We denote by
$\mC(d,a,b)$ the code obtained evaluating the vector space
$H^0(C_r,\pi^{*}(\Ol_{Y_r}(d)(-aP_\infty-bP_0)))$ on the set
$B:=Y_r(\F_{q^r})\setminus \{ P_\infty , P_0\}$.
\end{definition}

\begin{lemma}\label{reduc}
Let $\mC(d,a,b)$ be a code of Definition \ref{def_ab}. Assume $d>1$. If $b>d$
then $\mC(d,a,b)$ is strongly isometric to the code $\mC(d-1,a,0)$. Hence
$\mC(d,a,b)^\perp$ is strongly isometric to the code $\mC(d-1,a,0)^\perp$ (see
Remark \ref{duuu}).
\end{lemma}

\begin{proof}
 Keep in mind that $\mC(d,a,b)$ is the code obtained evaluating
$H^0(C_r,\pi^{*}(\Ol_{Y_r}(d)(-aP_\infty-bP_0)))$ on $B:=Y_r(\F_{q^r})\setminus
\{ P_\infty , P_0\}$. The curve $Y_r$ is smooth at $P_0$ and the tangent line to
$Y_r$ at $P_0$ has equation $y=0$. This line has contact order $c$ with $Y_r$
and does not
intersect $Y_r$ in any rational point different from $P_0$. Since $b>d$, if $f
\in
H^0(C_r,\pi^{*}(\Ol_{Y_r}(d)(-aP_\infty-bP_0)))$ then ($\pi^*$ is injective) $f$
is a degree $d$ form which is divided by $y$, the equation of the tangent line.
Hence the codes obtained evaluating
$H^0(C_r,\pi^{*}(\Ol_{Y_r}(d)(-aP_\infty-bP_0)))$ on $B$ and that obtained
evaluating $H^0(C_r,\pi^{*}(\Ol_{Y_r}(d-1)(-aP_\infty)))$ on $B$ are in fact
strongly isometric. 
\end{proof}

\begin{theorem}\label{teo_ab}
 Let $\mC(d,a,b)$ be as in Definition \ref{def_ab}. If $b>d$ then assume
$d>1$, set $b':=0$ and $d':=d-1$. Otherwise set $b':=b$ and $d':=d$. In any case
set $a':=a$. Assume $a' \in \{ 0,1 \}$.
\begin{enumerate}
 \item If $a'=0$ and $b'>0$ then the minimum distance of $\mC(d,a,b)^\perp$
is $d'+1$ and the number of the minimum-weight codewords of $\mC(d',0,b')^\perp$
is at least $(q^r-1)\binom{q^{r-1}-1}{d'+1}$.
\item If $b'=0$ and $a'=1$ then the minimum distance of $\mC(d,a,b)^\perp$
is $d'+1$ and the number of the minimum-weight codewords of $\mC(d',1,0)^\perp$
is exactly $$(q^r-1) \left[
(q^r-1)\binom{q^{r-1}}{d'+1}+\binom{q^{r-1}-1}{d'+1}\right].$$
\item If $a'=1$ and $b'>0$ then the minimum distance of $\mC(d,a,b)^\perp$
is $d'$ and the number of the minimum-weight codewords of $\mC(d,a,b)^\perp$
is exactly
$(q^r-1)\binom{q^{r-1}-1}{d'}$
\end{enumerate}
\end{theorem}

\begin{proof}
By Lemma \ref{reduc} we have $\mathcal{C}(d,a,b)\sim \mathcal{C}(d',a',b')$.
Hence
we can study the properties of the code $\mathcal{C}(d',a',b')$ without loss of
generality. If $a'=0$ and $b'>0$ then $d+1$ affine points of the curve different
from $P_0$ on the line of equation $x=0$ impose dependent conditions to
$H^0(C_r,\pi^*(\Ol_{Y_r}(d')(-b'P_0)))$ because the monomials $y,...,y^d \in
H^0(C_r,\pi^*(\Ol_{Y_r}(d')(-b'P_0)))$ and $y^{d+1}$ does not lie in this space.
If $a'=1$ and $b'=0$ then $d'+1$ affine points of the curve $Y_r$ on any line of
equation $x=\alpha$ ($\alpha \in \F_{q^r}$) and different from $P_0$ impose
dependent conditions to $H^0(C_r,\pi^*(\Ol_{Y_r}(d')(-P_\infty)))$ because
$1,y,...,y^{d-1}$ are in the basis of the vector space
$H^0(C_r,\pi^*(\Ol_{Y_r}(d')(-P_\infty)))$ and $y^{d}$ are not. If $a'=1$ and
$b'>0$ then any $d'$ affine points of the curve different from $P_0$ on the line
of equation $x=0$ impose dependent conditions to
$H^0(C_r,\pi^*(\Ol_{Y_r}(d')(-a'P_\infty-b'P_0)))$ because $y,...,y^{d-1} \in
H^0(C_r,\pi^*(\Ol_{Y_r}(d')(-a'P_\infty-b'P_0)))$ and $1,y^d$ do not. So in
cases (1) and (2) the dual minimum distance of $\mC(d',a',b')$ is at most
$d'+1$. In case (3) it is at most $d'$.
Let $S \subseteq B=Y_r(\F_{q^r})\setminus \{ P_0,P_\infty\}$ be the support of a
minimum-weight codeword of $\mC(d',a',b')^\perp$. The minimum distance of this
code is exactly $\sharp(S)$. The set $S$ imposes dependent conditions to the
space
$H^0(C_r,\pi^*(\Ol_{Y_r}(d')(-a'P_\infty-b'P_0)))$ and so it imposes dependent
conditions also to
$H^0(\Pro^2,\mathcal{I}_{a'P_\infty+b'P_0}(d'))$. It follows
$h^1(\Pro^2,\mathcal{I}_{a'P_\infty+b'P_0}(d'))>0$.
\begin{itemize} 
 \item Assume to be in case (1) or in case (2). Since $\sharp(S)+a'+b' \le
d'+1+1+b' \le 2d'+1$, Lemma \ref{u00.01} gives the existence of a line $L
\subseteq \Pro^2$ such that
$\deg(L \cap (a'P_\infty+b'P_0 \cup S)) \ge d'+2$. If $a'=0$ then
$\sharp(S)=d+1$. Otherwise $L$ has to be the tangent line to $Y_r$ at $P_0$,
which is absurd because $a'+b' \le d$. If $b'=0$ then $\sharp(S)=d+1$ because
$a'=1$. Hence the minimum distance of $\mC(d',a',b')^\perp$ is exactly $d'+1$.
If $b=0$ then any $d'+1$ affine points of the curve $Y_r$ different from $P_0$
on a vertical
line are in fact the support of a minimum weight codeword. There are
$\binom{q^{r-1}}{d'+1}$ such points on any such a line different from the line
of equation $x=0$ and $\binom{q^{r-1}-1}{d'+1}$ such points on the line of
equation $x=0$. If $a=0$ and $b>0$ then $d'+1$ points of the support of a
minimum-weight codeword of $\mC(d',0,b')^\perp$ lie on a line passing through
$P_0$.
\item Assume to be in case (3). As in the
previous part of the proof we get the existence of a line $L
\subseteq \Pro^2$ such that $\deg(L \cap a'P_\infty+b'P_0 \cup S)\ge d'+2$.
Since $\sharp(S) \le d'$ and $L$ cannot be the tangent line to $Y_r$ at $P_0$,
it follows
that $L$ is the line of equation $x=0$. The number of the
minimum-weight codewords trivially follows.
\end{itemize}
\end{proof}

\begin{remark}
 The hypothesis $a'+b' >0$ implicitly assumed in Theorem \ref{teo_ab} is in
fact not restrictive.
Indeed, if $a=b=0$ then, for any $d$, the code $\mC(d,0,0)$ is the code
$\mC(d,0)$
without the component corresponding to the evaluation at $P_0$.
\end{remark}

\begin{remark}
 The divisor of the rational function $y$ on the curve $Y_r$ is
$(y)=cP_0-cP_\infty$ (see \cite{geo}, Section 3). Hence we get the linear
equivalence $cP_0 \sim cP_\infty$. So if $d<q$ then Theorem \ref{teo_ab} is very
useful to
study two-point codes on norm-trace curves  (see Example \ref{ex2} below).
\end{remark}

The following is an interesting computational example.

\begin{example} \label{ex2}
 Set $r:=3$ and $q:=3$, so that $c=13$. Let us study the two-point code $\mC$ on
the curve $Y_3$ of equation
$$x^{13}=y^9+y^3+y$$ obtained evaluating the vector space $L(12P_\infty+11P_0)$
on
the set $B:=Y_r(\F_{q^r})\setminus \{ P_0,P_\infty\}$. Observe that
$12P_\infty \sim cP_\infty-P_\infty$ and that $11P_0 \sim cP_0-2P_0 \sim
cP_\infty-2P_0$. Hence
$$12P_\infty+11P_0 \sim 2cP_\infty-P_\infty-2P_0.$$ Set $d:=2$, $a:=1$ and
$b:=2$. Since $d<q$ the code $\mC$ is in fact strongly isometric to the code
$\mC(2,1,2)$ of
Definition \ref{def_ab} and its dual minimum distance is $2$.
Indeed, we can set $a':=a$, $b':=b$ and $d':=d$ and apply directly Theorem
\ref{teo_ab}. Let us study in details the code $\mC^\perp$. By using the linear
equivalence $12P_\infty+11P_0 \sim 26P_\infty-P_\infty-2P_0$ we have already
seen
that
$$L(12P_\infty+11P_0) \cong L(26P_\infty-P_\infty-2P_0)\cong
L(25P_\infty-2P_0).$$ The results of Section \ref{si} assure that we are not
changing the metric properties of the code $\mC^\perp$ by using these linear
equivalences. Apply the preliminary results of Section \ref{intr} to compute a
vector basis of $L(25P_\infty)$:
$$\{ 1,y,x,xy,x^2\}.$$
The rational function $x$ has a zero at $P_0$ of order $1$, while the rational
function $y$ has a zero at $P_0$ of order $c=13$ (see \cite{geo}, Section 3).
Hence $1,x \notin L(25P_\infty-2P_0)$ and 
$$\{ y,xy,x^2\} \subseteq L(25P_\infty-2P_0).$$
On the other hand, the Riemann-Roch space $L(25P_\infty-2P_0)$ is equal to
the vector space $$H^0(C_3,\pi^*(\Ol_{Y_r}(2)(-P_\infty-2P_0)))$$ (see Section
\ref{intr} again). Set $S:=2P_0$. The scheme $S$ imposes independent conditions
to the vector space $H^0(C_3,\pi^*(\Ol_{Y_r}(2)(-P_\infty)))$. Indeed, if it
imposes dependent
conditions to this space then it has to impose dependent conditions also to
$H^0(\Pro^2,\mathcal{O}_{\Pro^2}(2)(-P_\infty))$ (use the injectivity of the map
$\rho_{d,1}$ as in the proof of Theorem \ref{teo_a}). By Lemma \ref{u00.01}
there must exist a line $L \subseteq \Pro^2$ with the property $\deg(L \cap
(P_\infty \cup S)) \ge d+2=4$, which is absurd, because $\deg(S)=2$. This proves
that the dimension of $H^0(C_3,\pi^*(\Ol_{Y_r}(2)(-P_\infty-2P_0)))$ is
$\dim_{\F_{q^3}} L(25P_\infty) -2$. It follows that $\{
y,xy,x^2\}$ is in fact a basis of the
Riemann-Roch space $L(25P_\infty-2P_0)\cong L(12P_\infty+11P_0)$. So we have all
the explicit data needed to construct the code $\mathcal{C}^\perp$ in a
\texttt{Magma} environment. It can be checked that the minimum distance of
$\mC^\perp$ is in fact $2=d$. Hence the number of the minimum-weight codewords
of $\mC^\perp$
is exactly $728=26\cdot \binom{8}{2}$ (Theorem
\ref{teo_ab}).
\end{example}

\section{More general evaluation codes}

The result of Section \ref{one} and Section \ref{two} can be slightly extended
by using zero-dimensional schemes whose support is made of arbitrary affine
points of the curve $Y_r$.

\begin{definition}\label{def_E}
  Let $0<d<q^{r-1}-1$ be an integer. Choose a zero-dimensional subscheme $E
\subseteq \Pro^2$ such that $E_{red} \subseteq Y(\F_{q^r}) \cap \{ z=1\}$. We
denote by $\mC(d,E)$ the code obtained evaluating 
$H^0(C_r,\pi^{*}(\Ol_{Y_r}(d)(-E)))$ on the set $B:=Y_r(\F_{q^r})\setminus
(E_{red} \cap Y(\F_{q^r}))$.
\end{definition}

\begin{definition}
 Let $E \subseteq \Pro^2$ be a zero-dimensional scheme. Denote by $\mL$ the set
of the lines in $\Pro^2$ different from the line of equation $y=0$ and the line
at infinity of equation $z=0$. Denote by $\mV$ the set of the vertical lines in
$\Pro^2$. Define
$$m(E):=\max_{L \in \mL} \ \deg(E \cap L), \ \ \ \ m_\mV(E):=\max_{L \in \mV} \
\deg(E \cap L).$$
\end{definition}

\begin{theorem}
 Consider a $\mC(d,E)$ code as in Definition \ref{def_E}. Assume $\deg(E) \le
d$. The minimum distance of $\mC(d,E)^\perp$ is at least $d+2-m(E)$. If
$m(E)=m_\mV(E)$ then the minimum distance of $\mC(d,E)^\perp$ is exactly
$d+2-m_\mV(E)$ and the number of the minimum-weight codewords of
$\mC(d,E)^\perp$ is at least
$$(q^r-1)\left[ (q^{r}-1)\binom{q^{r-1}}{m_\mV(E)} + \binom{q^{r-1}-1}{m_\mV(E)}
\right].$$
\end{theorem}

\begin{proof}
 If $E=\emptyset$ then the thesis trivially follows from Theorem \ref{te1}.
Assume $E \neq \emptyset$. There obviously exists a vertical line $L$ such that
$\deg(L \cap E) \ge 1$ and, by definition of $m_\mV(E)$, $\deg(L \cap E) \le
m_\mV(E)$. The
scheme $L \cap E$ is reduced. Indeed, if there exists a point $P\in Y(\F_{q^r})
\cap \{ z=1 \}$ such that $2P \subseteq L \cap E$ then $L$ has to be the tangent
line to $Y_r$ at $P$. The tangent line to $Y_r$ at
$P=(\overline{x}:\overline{y}:\overline{z})$ has equation
$$\overline{x}^{c-1} x-\overline{z}^{c-1}y+\frac{\partial Y_r}{\partial
z}(\overline{x}:\overline{y}:\overline{z}) z=0.$$ 

Since $\overline{z} \neq 0$, this line cannot be vertical, a contradiction. Let
$L$ be a vertical line
which realizes the maximum in the definition of $m_\mV(E)$. Set
$A:=E \cap L$ and observe that $\deg(A)=m_\mV(E)$. Choose $d+2-m_\mV(E)$
distinct points in $L \setminus A$ and denote by $S$ their union (as a
zero-dimensional scheme). Since $d<q^{r-1}-1<c$, the restriction map
$$\rho_d:H^0(\Pro^2,\mathcal{O}_{\Pro^2}(d)) \to H^0(C_r,\pi^*(\Ol_{Y_r}(d)))$$
is injective. As in the proof of Theorem \ref{te1}, the set $S \cup A$ (whose
degree is $d+2$) imposes
dependent conditions to $H^0(C_r,\pi^*(\Ol_{Y_r}(d)))$. On the other hand, the
set $A$ imposes independent conditions to this space. Indeed, if it imposes
dependent conditions, then by Lemma \ref{u00.01} there must exist a line $R
\subseteq
\Pro^2$ such that $\deg(R \cap A) \ge d+2$. Since $\deg(A) \le \deg(E)$, this
leads to a contradiction. It follows that $S=(S \cup A)\setminus A$ imposes
dependent conditions to the space $H^0(C_r,\pi^*(\Ol_{Y_r}(d)(-A)))$. In
particular, it imposes dependent conditions to
$H^0(C_r,\pi^*(\Ol_{Y_r}(d)(-E)))$. In other words, $S$ contains the support of
a codeword of $\mathcal{C}(d,E)^\perp$. Hence the minimum distance, say
$\delta$, of  $\mathcal{C}(d,E)^\perp$ has to verify $\delta \le
\sharp(S)=d+2-m_\mV(E)$. Assume that $S \subseteq B=Y_r(\F_{q^r})\setminus
(E_{red} \cap Y(\F_{q^r}))$ is the support of a minimum-weight codeword of
$\mC(d,E)^\perp$. The minimum distance of $\mC(d,E)^\perp$ is exactly
$\sharp(S)$ and
$\sharp(S) \le d+2-m_\mV(E)$. Since the restriction map
$$\rho_{d,E}:H^0(\Pro^2,\mathcal{O}_{\Pro^2}(d)(-E)) \to
H^0(C_r,\pi^*(\Ol_{Y_r}(d)(-E)))$$
is injective ($d<q^{r-1}-1 < c$ by assumption), the set $S$ has to impose
dependent conditions to the space
$H^0(\Pro^2,\mathcal{O}_{\Pro^2}(d)(-E))$ and in particular we have
$h^1(\Pro^2,\mathcal{I}_{E\cup S}(d))>0$. Since $\deg(E \cup S) \le
d+d+2-m_\mV(E)
\le 2d+1$ we can apply Lemma \ref{u00.01} to get the existence of a line $R
\subseteq \Pro^2$ such that $\deg(R \cap (E \cup S)) \ge d+2$. The line $R$
cannot be neither the line of equation $y=0$, or the line at infinity. It
follows $\sharp(S) \ge \deg(R\cap S) \ge d+2-\deg(R \cap E) \ge d+2-m(E)$. This
proves that the minimum distance of $\mC(d,E)^\perp$ is at least $d+2-m(E)$. If
$m(E)=m_\mV(E)$ then we have in fact proved that the minimum distance of
$\mC(d,E)^\perp$
is exactly $d+2-m(E)$ and any $d+2-m(E)$ points on a vertical line avoiding the
support of $E$ are the support of a minimum-weight codeword. The theorem is
proved.
\end{proof}

\section{Remarks on the dimension of two-point codes on norm-trace curves}

Let $m,n$ be integers such that $m+n>0$. Write $m=d_1c-a$ and
$n=d_2c-b$ with $0 \le a,b \le c-1$. Set $d:=d_1+d_2$. On the curve $Y_r$ it
holds the linear equivalence $cP_0 \sim cP_\infty$ and so we get
$$mP_\infty+nP_0 \sim dcP_\infty-aP_\infty-bP_0.$$
If $d<q$ then the two-point code on $Y_r$ obtained evaluating the rational
functions in the Riemann-Roch space $L(mP_\infty+nP_0)$ on the set
$B:=Y(\F_{q^r}) \setminus \{ P_\infty, P_0\}$ is in fact the code obtained
evaluating the vector space $H^0(C_r,
\pi^*(\mathcal{O}_{Y_r}(d)(-aP_\infty-bP_0)))$ on $B$, i.e.
$\mC(d,a,b)$ (see Definition \ref{def_ab}). 

\begin{lemma}\label{dim}
 Let $0<d<q$, $0 \le a,b \le c-1$ be integers with $b>0$. The dimension of
$\mC(d,a,b)$ is $h^0(Y_r, \mathcal{O}_{Y_r}(d)(-aP_\infty-bP_0))$.
\end{lemma}

\begin{proof}
The point $P_\infty$ is a singular point. We denote by $\pi:C_r \to Y_r$ the
normalization of the norm-trace curve $Y_r$. The map $\pi$ is known to be a
bijection. Let $Q_0:=\pi^{-1}(P_0)$ and $Q_\infty:=\pi^{-1}(P_\infty)$, which is
a nonsingular point of $C_r$. Since $d<q$ it follows that $\mC(d,a,b)$ is the
code obtained evaluating the vector space $L(dcQ_\infty-aQ_\infty-bQ_0)$ on the
set $\pi^{-1}(B)$. Since $|\pi^{-1}(B)|=q^{2r-1}-1$ we have
$$dc-a-b-\deg(\pi^{-1}(B))<0.$$ It follows that the kernel of the evaluation map
$\mbox{ev}: L(dcQ_\infty-aQ_\infty-bQ_0) \to \F_{q^r}^{|B|}$ is a
zero-dimensional vector
space and so the image of $\mbox{ev}$ (which is exactly $\mC(d,a,b)$) has
dimension
$\ell(dcQ_\infty-aQ_\infty-bQ_0)=h^0(C_r,\pi^*(\Ol_{Y_r}
(d)(-aQ_\infty-bQ_0)))=h^0(Y_r,\Ol_{Y_r}(d)(-aP_\infty-bP_0))$.
\end{proof}

\begin{remark}
 The case $b=0$ is not of interest. Indeed, a $\mathcal{C}(d,a,0)$ code is a
shortening of a $\mC(d,a)$ code (see Definition \ref{def_a}).
\end{remark}

\begin{lemma} \label{-1}
  Let $0<d<q$, $0 \le a,b \le c-1$ be integers with $b>0$. If $b>d$ then
$\mC(d,a,b)$ has dimension $h^0(Y_r,\Ol_{Y_r}(d-1)(-aP_\infty))$.
\end{lemma}

\begin{proof}
 By Lemma \ref{dim} it is enough to prove that
$h^0(Y_r,\Ol_{Y_r}(d)(-aP_\infty-bP_0))=h^0(Y_r,\Ol_{Y_r}(d-1)(-aP_\infty)$. A
form $f \in H^0(Y_r,\Ol_{Y_r}(d)(-aP_\infty-bP_0))$ is a degree
 $d$ homogeneous polynomial on the curve $Y_r$ vanishing at $P_0$ with order at
least
$b$.
 Since $P_0$ is a nonsingular point of the curve $Y_r$, $f$ is divided by the
 equation of the tangent space to $Y_r$ at $P_0$, which is $y=0$. The division
 by $y$ defines in fact an isomorphism of vector spaces
 $$H^0(Y_r,\Ol_{Y_r}(d)(-aP_\infty-bP_0)) \rightarrow
 H^0(Y_r,\Ol_{Y_r}
 (d-1)(-aP_\infty),$$ whose inverse is the multiplication by $y$ (the
 tangent line to $Y_r$ at $P_0$ has contact order $c \ge b$).
\end{proof}

\begin{notation}
 The dimension of the Riemann-Roch space $L(sP_\infty)$ on $Y_r$ will be denoted
by
$N(s)$. If $0 \le s \le cq^r$ then $N(s)$ is the number of the
pairs $(i,j) \in \N^2$ such that $$i<q^r, \ \ j<q^{r-1}, \ \
iq^{r-1}+jc \le s.$$ The basis for $L(sP_\infty)$ made of the monomials
$x^iy^j$ ($i,j$ with the cited properties) will be denoted by $\mB_s$.
\end{notation}

\begin{proposition}\label{red}
 Let $0<d<q$, $0 \le a \le c-1$ and $0 \le b \le d$ be integers with $b>0$. Set
$s:=dc-a$. Then $h^0(Y_r,\Ol_{Y_r}(d)(-aP_\infty-bP_0))=\ell(s)-b$.
\end{proposition}
\begin{proof}
 First of all, let us consider the trivial inclusion of Riemann-Roch spaces
 $L(dcP_\infty-aP_\infty-bP_0) \subseteq L(dcP_\infty-aP_\infty)$. We have
in any case $\ell(dcP_\infty-aP_\infty-bP_0) \ge
\ell(dcP_\infty-aP_\infty)-b$.
 Since $b \le d$ in the
 basis $\mathcal{B}_{s}$ appear the monomials
 $1,x,...,x^{b-1}$. These rational functions are linearly independent and do
not
 lie in $L(dcP_\infty-aP_\infty-bP_0)$, because $x$ has a zero of order one
at $P_0$. Hence the dimension of this space is
 exactly
 $N(s)-b$. Moreover, it is spanned by the monomials in
 $\mathcal{B}_{s} \cap L(dcP_\infty-aP_\infty-bP_0)$.
\end{proof}

\begin{corollary}
  Let $0<d<q$, $0 \le a,b \le c-1$ be integers with $b>0$. Set $s:=dc-a$.
\begin{enumerate}
 \item If $b \le d$ then the dimension of $\mC(d,a,b)$ is $N(s)-b$.
\item If $b>d$ then the dimension of $\mC(d,a,b)$ is $N(s-c)$.
\end{enumerate}
\end{corollary}
\begin{proof}
 If $b \le d$ then apply Proposition \ref{red}. If $b>d$ then use Lemma
\ref{-1}.
\end{proof}

\section*{Acknowledgment} The authors would like to thank the Referees for
suggestions that improved the presentation of this work.

\bibliographystyle{elsarticle-num}

\end{document}